\documentclass[11pt]{article}
\usepackage{makeidx}
\usepackage{amsfonts}
\usepackage{amsmath}
\usepackage{amssymb}
\usepackage{bbold}
\usepackage{extarrows}
\usepackage{graphicx}
\usepackage{bbold}
\usepackage[usenames,dvipsnames]{xcolor}
\usepackage{hyperref}
\usepackage{palatino}

\hypersetup{colorlinks=true,
linkcolor=blue, citecolor=green, filecolor=black, urlcolor=black }
\providecommand{\U}[1]{\protect\rule{.1in}{.1in}}
\newtheorem{theorem}{Theorem}

\newtheorem{definition}[theorem]{Definition}

\newtheorem{lemma}[theorem]{Lemma}

\newtheorem{proposition}[theorem]{Proposition}
\newtheorem{remark}[theorem]{Remark}

\newenvironment{proof}[1][Proof]{\noindent\textbf{#1.} }{\ \rule{0.5em}{0.5em}}
\renewcommand{\thefootnote}{\fnsymbol{footnote}}

\begin{document}

\author{Khaled Bahlali$\;^{a,\S ,1}$, Lucian Maticiuc$\;^{b,2}$, Adrian Z%
\u{a}linescu$\;^{b,2}$\bigskip \\
{\small $^{a}$~Universit\'{e} de Toulon, IMATH, EA 2134, 83957 La Garde
Cedex, France.}\\
{\small $^{b}$~Faculty of Mathematics, \textquotedblleft Alexandru Ioan
Cuza\textquotedblright\ University,}\\
{\small Carol I Blvd., no. 9, Iasi, 700506, Romania.}}
\title{Penalization method for a nonlinear Neumann PDE via weak solutions of
reflected {SDEs}$^{\ast }$}
\maketitle

\begin{abstract}
In this paper we prove an approximation result for the viscosity solution of
a system of semi-linear partial differential equations with continuous
coefficients and nonlinear Neumann boundary condition. The approximation we
use is based on a penalization method and our approach is probabilistic. We
prove the weak uniqueness of the solution for the reflected stochastic
differential equation and we approximate it (in law) by a sequence of
solutions of stochastic differential equations with penalized terms. Using
then a suitable generalized backward stochastic differential equation and
the uniqueness of the reflected stochastic differential equation, we prove
the existence of a continuous function, given by a probabilistic
representation, which is a viscosity solution of the considered partial
differential equation. In addition, this solution is approximated by
solutions of penalized partial differential equations.
\end{abstract}

\footnotetext[1]{{\scriptsize A part of this paper was presented, by
one of the authors, at \textquotedblleft Workshop on Stochastic
Analysis and Applications\textquotedblright , Marrakech - EL Kelaa
Mgouna, April 9 - 14, 2012.}} \footnotetext[4]{{\scriptsize CNRS,
LATP, CMI, Aix Marseille
Universit\'{e}, 39 rue Joliot-Curie,13453 Marseille (2013-2014).}} %
\renewcommand{\thefootnote}{\arabic{footnote}}\footnotetext[1]{{\scriptsize %
The work of this author was partially supported by PHC Volubilis MA/10/224
and PHC Tassili 13MDU887.}} \footnotetext[2]{{\scriptsize The work of this
author was supported by IDEAS no. 241/05.10.2011 and POSDRU/89/1.5/S/49944.}}%
\renewcommand{\thefootnote}{\fnsymbol{footnote}} \footnotetext{{\scriptsize %
E-mail addresses: bahlali@univ-tln.fr (Khaled Bahlali),
lucian.maticiuc@ymail.com (Lucian\ Maticiuc), adrian.zalinescu@gmail.com
(Adrian Z\u{a}linescu)}}

\textbf{AMS Classification subjects:} 60H99, 60H30, 35K61.\medskip

\textbf{Keywords or phrases: }Reflecting stochastic differential equation;
Penalization method; Weak solution; Jakubowski S-topology; Backward
stochastic differential equations.

\section{Introduction}

Let $G$ be a $C^{2}$ convex, open and bounded set from
$\mathbb{R}^{d}$ and for $\left( t,x\right) \in \left[ 0,T\right]
\times \bar{G}$ we consider the following
reflecting stochastic differential equation (SDE for short)%
\begin{equation*}
{X_{s}+K}_{s}{=x+\displaystyle\int_{t}^{s}b({X_{r})}dr+\displaystyle%
\int_{t}^{s}\sigma ({X_{r})}dW_{r}}\,,\;s\in \left[ t,T\right] ,
\end{equation*}%
with ${K}$ a bounded variation process such that for any $s\in \left[ t,T\right] $, $K_{s}={\displaystyle%
\int_{t}^{s}}\nabla \ell ({X_{r})}d\left\vert K\right\vert _{\left[ t,r%
\right] }$ and $\left\vert K\right\vert _{\left[ t,s\right] }={\displaystyle%
\int_{t}^{s}}\mathbb{1}_{\{X_{r}\in \partial G\}}d\left\vert K\right\vert _{%
\left[ t,r\right] }$ , where the notation $\left\vert K\right\vert _{\left[
t,s\right] }$ stands for the total variation of $K$ on the interval $\left[
t,s\right] $.

The coefficients $b$ and $\sigma$ are supposed to be only bounded continuous
on $\mathbb{R}^{d}$ and $\sigma\sigma^{\ast}$ uniformly elliptic. The first
main purpose is to prove that the weak solution $\left( X,K\right) $ is
approximated in law (in the space of continuous functions) by the solutions
of the non-reflecting SDE
\begin{equation*}
{X_{s}^{n}=x+\int_{t}^{s}}\left[ {b({X_{r}^{n})-}n}({X_{r}^{n}}-\pi_{\bar{G}%
}({X_{r}^{n}))}\right] {dr+\int_{t}^{s}\sigma({X_{r}^{n})}dW_{r}} , \ \ \
s\in\left[ t,T\right] ,
\end{equation*}
where $\pi_{\bar{G}}$ is the projection operator. Since for $n\rightarrow
\infty$ the term $K_{s}^{n}:={n\int_{t}^{s}}({X_{r}^{n}}-\pi_{\bar{G}}({%
X_{r}^{n}))dr}$ forces the solution ${X^{n}}$ to remain near the domain, the
above equation is called SDE with penalization term.

The case where $b$ and $\sigma $ are Lipschitz has been considered by Lions,
Menaldi and Sznitman in \cite{li-me-sz/81} and by Menaldi in \cite{me/83}
where they have proven that $\mathbb{E}(\sup_{s\in \lbrack 0,\
T]}|X_{s}^{n}-X_{s}|)\longrightarrow 0$, as $n\rightarrow \infty $. Note
that Lions and Sznitman have shown, using Skorohod problem, the existence of
a weak solution for the SDE with normal reflection to a (non-necessarily
convex) domain. The case of reflecting SDE with jumps has been treated by \L %
aukajtys and S\l omi\'{n}ski in \cite{la-sl/03} in the Lipschitz case; the
same authors have extended in \cite{la-sl/12} these results to the case
where the coefficient of the reflecting equation is only continuous.
In these two papers it is proven that the approximating sequence $\left(
X^{n}\right) _{n}$ is tight with respect to the $\mathrm{S}$-topology,
introduced by Jakubowski in \cite{ja/97} on the space $\mathcal{D}\left(
\mathbb{R}_{+},\mathbb{R}^{d}\right) $ of c\`{a}dl\`{a}g $\mathbb{R}^{d}$%
-valued functions. {Assuming the weak (in law) uniqueness of the limiting
reflected diffusion $X$, they prove in \cite{la-sl/12} that $X^{n}$ $\mathrm{%
S}$-converges weakly to }${X}$. We mention that $\left( X^{n}\right) _{n}$
may not be relatively compact with respect to the Skorohod topology $J_{1}$.

In contrast to \cite{la-sl/12}, we can not simply assume the uniqueness in
law of the limit $X,$ and the weak $\mathrm{\ S}$-convergence of $X^{n}$ to $%
X$ is not sufficient to our goal. In our framework, we need to show the
uniqueness in law of the couple $(X,K)$ and that the convergence in law of
the sequence $(X^{n},K^{n})$ to $(X,K)$ holds with respect to uniform
topology.

The first main result of our paper will be the weak uniqueness of the
solution $\left( X,K\right) $, together with the convergence in law (in the
space of continuous functions) of the penalized diffusion to the reflected
diffusion $X$ and the continuity with respect to the initial data.

Subsequently, using a proper generalized BSDE, we deduce (as a second main
result) an approximation result for a continuous viscosity solution of the
system of semi-linear partial differential equations (PDEs for short) with a
nonlinear Neumann boundary condition
\begin{equation*}
\begin{array}{l}
\dfrac{\partial u_{i}}{\partial t}\left( t,x\right) +Lu_{i}\left( t,x\right)
+f_{i}\left( t,x,u\left( t,x\right) \right) =0,\;\forall \left( t,x\right)
\in \lbrack 0,T]\times G,\medskip \\
\dfrac{\partial u_{i}}{\partial n}\left( t,x\right) +h_{i}\left( t,x,u\left(
t,x\right) \right) =0,\;\forall \left( t,x\right) \in \lbrack 0,T]\times
\partial G,\medskip \\
u_{i}\left( T,x\right) ={g}_{i}{(x}),\;\forall x\in G,\;i=\overline{1,k},%
\end{array}%
\end{equation*}%
where $L$ is the infinitesimal generator of the diffusion $X$, defined by
\begin{equation*}
L=\frac{1}{2}\sum_{i,j}(\sigma (\cdot )\,\sigma ^{\ast }(\cdot ))_{ij}(\cdot
)\frac{{\partial }^{2}}{\partial x_{i}\partial x_{j}}+\sum_{i}b_{i}(\cdot )%
\frac{\partial }{\partial x_{i}}~,
\end{equation*}%
and $\partial u_{i}/\partial n$ is the outward normal derivative of $u_{i}$
on the boundary of the domain.

Boufoussi and Van Casteren have established in \cite{bo-ca/04} a similar
result, but in the case where the coefficients $b$ and $\sigma $ are
uniformly Lipschitz.

We mention that the class of BSDEs involving a Stieltjes integral with
respect to the continuous increasing process $\left\vert K\right\vert _{%
\left[ t,s\right] }$ was studied first in \cite{pa-zh/98} by Pardoux and
Zhang; the authors provided a probabilistic representation for the viscosity
solution of a Neumann boundary partial differential equation. {It should be
mentioned that the continuity of the viscosity solution is rather hard to
prove in our frame. In fact, }this property essentially use the continuity
with respect to initial data of the solution of our BSDE. {We develop here a
more natural method based on the uniqueness in law of the solution $(X,K,Y)$
of the reflected SDE-BSDE and on the continuity property. Similar techniques
were developed, in the non reflected case, in \cite{ba-el-pa/09}, {\ but in
our situation the proof is more delicate. The} difficulty} is due to the
presence of the reflection process $K$ in the forward component and the
generalized part in the backward component.$\smallskip $

Throughout this paper we use different types of convergence defined as
follows: for the processes $\left( Y^{n}\right) _{n}$ and $Y$, by ${Y^{n}}%
\xrightarrow[u]{\;\;\;*\; \;\;}Y$ we denote the convergence in law with
respect to the uniform topology, by ${Y^{n}%
\xrightarrow[J_{1}]{\;\;\;*\;
\;\;}Y^{n}}$ we mean the convergence in law with respect to the Skorohod
topology $J_{1}$ and by ${Y^{n}}\xrightarrow[\mathrm{S}]{\;\;\;*\; \;\;}Y$
we understand the weak convergence considered in \textrm{S}-topology.

The paper is organized as follows: in the next section we give the
assumptions, we formulate the problem and we state the two main results. The
third section is devoted to the proof of the first main result (proof of the
convergence in law of $\left( X^{n},K^{n}\right) $ to $\left( X,K\right) $
as $n\rightarrow\infty$ and the continuous dependence with respect to the
initial data). In Section 4 the generalized BSDEs are introduced, the
continuity with respect to the initial data is obtained and we prove the
approximation result for the PDE introduced above.

\section{Formulation of the problem; the main results}

Let $G$ be a $C^{2}$ convex, open and bounded set from $\mathbb{R}^{d}$ and
we suppose that there exists a function $\ell\in C_{b}^{2}(\mathbb{R}^{d})$
such that%
\begin{equation*}
G=\{x\in\mathbb{R}^{d}:\ell\left( x\right) <0\},\;\partial G=\{x\in \mathbb{R%
}^{d}:\ell\left( x\right) =0\},
\end{equation*}
and, for all $x\in\partial G$, $\nabla\ell\left( x\right) $ is the unit
outward normal to $\partial G$.

In order to define the approximation procedure we shall introduce the
penalization term. Let $p:\mathbb{R}^{d}\rightarrow\mathbb{R}_{+}$ be given
by $p\left( x\right) =\mathrm{dist}^{2}(x,\bar{G})$.

Without restriction of generality we can choose $\ell$ such that%
\begin{equation*}
\left\langle \nabla\ell\left( x\right) ,\delta\left( x\right) \right\rangle
\geq0,\;\forall x\in\mathbb{R}^{d},
\end{equation*}
where $\delta\left( x\right) :=\nabla p\left( x\right) $ is called the
penalization term.

It can be shown that $p$ is of class $C^{1}$ on $\mathbb{R}^{d}$ with%
\begin{equation*}
\frac{1}{2}\delta\left( x\right) =\frac{1}{2}\nabla(\mathrm{dist}^{2}(x,\bar{%
G}))=x-\pi_{\bar{G}}\left( x\right) ,\;\forall x\in\mathbb{R}^{d},
\end{equation*}
where $\pi_{\bar{G}}\left( x\right) $ is the projection of $x$ on $\bar{G}.$
It is clear that $\delta$ is a Lipschitz function.

On the other hand, $x\mapsto\mathrm{dist}^{2}\left( x,\bar{G}\right) $ is a
convex function and therefore%
\begin{equation}
\left\langle z-x,\delta\left( x\right) \right\rangle \leq0,\;\forall x\in%
\mathbb{R}^{d},\;\forall z\in\bar{G}.  \label{subdif ineq}
\end{equation}
Let $T>0$ and suppose that:

\begin{itemize}
\item[\textrm{(A}$_{1}$\textrm{)}] $b:\mathbb{R}^{d}\rightarrow\mathbb{R}%
^{d} $\ \textit{and }$\sigma:\mathbb{R}^{d}\rightarrow\mathbb{R}^{d\times r}$%
\textit{\ are bounded continuous functions.}
\end{itemize}

\begin{remark}
In fact we can assume that the functions $b$ and $\sigma$ have sublinear
growth but, for the simplicity of the calculus, we will work with assumption
\textrm{(A}$_{1}$\textrm{).}
\end{remark}

\begin{itemize}
\item[\textrm{(A}$_{2}$\textrm{)}] \textit{the matrix }$\sigma\sigma^{\ast}$%
\textit{\ is uniformly elliptic}, i.e. there exists $\alpha_{0}>0$ such that
for all $x\in\mathbb{R}^{d}$, $\left( \sigma\sigma^{\ast}\right) \left(
x\right) \geq\alpha_{0}~I$.
\end{itemize}

Moreover, there exist some positive constants $C_{i}$, $i=\overline{1,2}$, $%
\alpha\in\mathbb{R}$, $\beta\in\mathbb{R}_{+}^{\ast}$ and $q\geq1$ such that

\begin{itemize}
\item[\textrm{(A}$_{3}$\textrm{)}] $f,h:\left[ 0,T\right] \times \mathbb{R}%
^{d}\times\mathbb{R}^{k}\rightarrow\mathbb{R}^{k}$\textit{\ and }$g:\mathbb{R%
}^{d}\rightarrow\mathbb{R}^{k}$ \textit{are continuous functions and, for
all }$x,x^{\prime}\in\mathbb{R}^{d}$, $y,y^{\prime}\in\mathbb{R}^{k}$, $%
t,t^{\prime}\in\left[ 0,T\right] \,,$%
\begin{equation}
\begin{array}{rl}
\left( i\right) & \left\langle
y^{\prime}-y,f(t,x,y^{\prime})-f(t,x,y)\right\rangle \leq\alpha\,\left\vert
y^{\prime}-y\right\vert ^{2},\medskip \\
\left( ii\right) & \left\vert h(t^{\prime},x^{\prime},y^{\prime
})-h(t,x,y)\right\vert \leq\beta\left( \left\vert t^{\prime}-t\right\vert
+\left\vert x^{\prime}-x\right\vert +\left\vert y^{\prime}-y\right\vert
\right) ,\medskip \\
\left( iii\right) & \left\vert f(t,x,y)\right\vert +\left\vert
h(t,x,y)\right\vert \leq C_{1}\left( 1+\left\vert y\right\vert \right)
,\medskip \\
\left( iv\right) & \left\vert g(x)\right\vert \leq C_{2}\left( 1+\left\vert
x\right\vert ^{q}\right) \,.%
\end{array}
\label{assumpt f,g,h}
\end{equation}
\end{itemize}

Let us consider the following system of semi-linear PDEs considered on the
whole space:%
\begin{equation}
\left\{
\begin{array}{r}
\dfrac{\partial u_{i}^{n}}{\partial t}\left( t,x\right) +Lu_{i}^{n}\left(
t,x\right) +f_{i}\left( t,x,u^{n}\left( t,x\right) \right) -{\left\langle
\nabla u_{i}^{n}(t,{x)},{n\delta}({x)}\right\rangle }\medskip \\
-{\left\langle \nabla\ell({x)},{n\delta}({x)}\right\rangle ~}h_{i}\left(
t,x,u^{n}\left( t,x\right) \right) =0\medskip \\
\multicolumn{1}{l}{u_{i}^{n}\left( T,x\right) ={g}_{i}{(x}),\;\forall\left(
t,x\right) \in\lbrack0,T]\times\mathbb{R}^{d},\;i=\overline{1,k}~,}%
\end{array}
\right.  \label{PDE Neumann approx}
\end{equation}
and the next semi-linear PDE considered with Neumann boundary conditions:%
\begin{equation}
\left\{
\begin{array}{l}
\dfrac{\partial u_{i}}{\partial t}\left( t,x\right) +Lu_{i}\left( t,x\right)
+f_{i}\left( t,x,u\left( t,x\right) \right) =0,\;\forall \left( t,x\right)
\in\lbrack0,T]\times G,\medskip \\
\dfrac{\partial u_{i}}{\partial n}\left( t,x\right) +h_{i}\left( t,x,u\left(
t,x\right) \right) =0,\;\forall\left( t,x\right) \in
\lbrack0,T]\times\partial G,\medskip \\
u_{i}\left( T,x\right) ={g}_{i}{(x}),\;\forall x\in G,\;i=\overline{1,k},%
\end{array}
\right.  \label{PDE Neumann}
\end{equation}
where $L$ is the second order partial differential operator%
\begin{equation*}
L=\frac{1}{2}\sum_{i,j}(\sigma(\cdot)\,\sigma^{\ast}(\cdot))_{ij}(\cdot )%
\frac{{\partial}^{2}}{\partial x_{i}\partial x_{j}}+\sum_{i}b_{i}(\cdot )%
\frac{\partial}{\partial x_{i}}~,
\end{equation*}
and, for any $x\in\partial G$%
\begin{equation*}
\dfrac{\partial u_{i}}{\partial n}\left( t,x\right) =\left\langle \nabla
\ell\left( x\right) ,\nabla u_{i}\left( t,x\right) \right\rangle
\end{equation*}
is the exterior normal derivative in $x\in\partial G.$

Our goal is to establish a connection between the viscosity solutions for (%
\ref{PDE Neumann approx}) and (\ref{PDE Neumann}) respectively. The proof
will be given using a probabilistic approach. Therefore we start by studying
an SDE with reflecting boundary condition and then we associate a
corresponding backward SDE. Since the coefficients of the forward equation
are merely continuous, our setting is that of weak formulation of solutions.$%
\bigskip$

For $\left( t,x\right) \in\left[ 0,T\right] \times\bar{G}$ we consider the
following stochastic differential equation with reflecting boundary
condition:%
\begin{equation}
\begin{array}{rl}
\left( i\right) & {X_{s}^{t,x}+K}_{s}^{t,x}{=x+\displaystyle\int_{t}^{s}b({%
X_{r}^{t,x})}dr+\displaystyle\int_{t}^{s}\sigma({X_{r}^{t,x})}dW_{r}}%
\,,\medskip \\
\left( ii\right) & K_{s}^{t,x}={\displaystyle\int_{t}^{s}}\nabla\ell ({%
X_{r}^{t,x})}d\left\vert K^{t,x}\right\vert _{\left[ t,r\right] }\,,\medskip
\\
\left( iii\right) & \left\vert K^{t,x}\right\vert _{\left[ t,s\right] }={%
\displaystyle\int_{t}^{s}}\mathbb{1}_{\{X_{r}^{t,x}\in\partial
G\}}d\left\vert K^{t,x}\right\vert _{\left[ t,r\right] }\,,\;\forall s\in%
\left[ t,T\right] ,%
\end{array}
\label{reflected SDE}
\end{equation}
where $\left\vert K^{t,x}\right\vert _{\left[ t,s\right] }$ is the the total
variation of $K^{t,x}$ on the interval $\left[ s,t\right] \footnote{%
For $0\leq s<t\leq T$, the total variation of $Y$ on $\left[ s,t\right] $ is
given by $\left\vert Y\right\vert _{\left[ s,t\right] }\left( \omega\right)
=\sup\limits_{\Delta}\Big\{\sum \limits_{i=0}^{n-1}|Y_{t_{i+1}}\left(
\omega\right) -Y_{t_{i}}\left( \omega\right) |\Big\},$ where $%
\Delta:s=t_{0}<t_{1}<\cdots<t_{n}=t$ is a partition of the interval $\left[
s,t\right] $.}$.

We denote by $k_{s}^{t,x}$ the continuous increasing process defined by $%
k_{s}^{t,x}:=\left\vert K^{t,x}\right\vert _{\left[ t,s\right] }$. It
follows that%
\begin{equation}
k_{s}^{t,x}={\int_{t}^{s}}\left\langle \nabla\ell({X_{r}^{t,x})}%
,dK_{r}^{t,x}\right\rangle \,.  \label{def k}
\end{equation}
Using the penalization term $\delta$ we can define the approximation
procedure for the reflected diffusion $X$.

Under assumption $\mathrm{(A}_{1}\mathrm{)}$ we know that (see, e.g., \cite[%
Theorem 5.4.22]{ka-sh/88}), for each $n\in \mathbb{N}^{\ast }$, there exists
a weak solution of the following penalized SDE%
\begin{equation}
{X_{s}^{t,x,n}=x+\int_{t}^{s}}\left[ {b({X_{r}^{t,x,n})-}n\delta }({%
X_{r}^{t,x,n})}\right] {dr+\int_{t}^{s}\sigma ({X_{r}^{t,x,n})}dW_{r}}%
\,,\;\forall s\in \left[ t,T\right] .  \label{penalized reflected SDE}
\end{equation}%
Let%
\begin{equation}
\begin{array}{l}
K_{s}^{t,x,n}:=\displaystyle{\int_{t}^{s}n\delta }({X_{r}^{t,x,n})}%
dr\,,\medskip \\
{{k}_{s}^{t,x,n}:=\displaystyle{\int_{t}^{s}}\left\langle \nabla \ell ({%
X_{r}^{t,x,n})},dK_{r}^{t,x,n}\right\rangle ~,}\;\forall s\in \left[ t,T%
\right] .%
\end{array}
\label{def k^n}
\end{equation}%
We mention that (see, e.g., \cite{ka-sh/88}) the solution process $({{%
X_{s}^{t,x,n})}}_{s\in \left[ t,T\right] }$ is unique in law under the
supplementary assumption $\mathrm{(A}_{2}\mathrm{)}$.

Here and subsequently, we shall denote by ${V}$ and ${V}^{n}$:%
\begin{equation}
\begin{array}{l}
{V}_{s}^{t,x}:={x+\displaystyle\int_{t}^{s}b({X_{r}^{t,x,n})}dr+\displaystyle%
\int_{t}^{s}\sigma({X_{r}^{t,x,n})}dW_{r}}\,,\;\medskip \\
{V}_{s}^{t,x,n}:={x+\displaystyle\int_{t}^{s}b({X_{r}^{t,x,n})}dr+%
\displaystyle\int_{t}^{s}\sigma({X_{r}^{t,x,n})}dW_{r}}\,,\;\forall s\in%
\left[ t,T\right] .%
\end{array}
\label{def V^n}
\end{equation}
Hence (\ref{reflected SDE}) and (\ref{penalized reflected SDE}) become
respectively%
\begin{equation*}
X_{s}^{t,x}+K_{s}^{t,x}=V_{s}^{t,x}\quad\text{and}\quad
X_{s}^{t,x,n}+K_{s}^{t,x,n}=V_{s}^{t,x,n}~,\;\forall s\in\left[ t,T\right] .
\end{equation*}

\begin{definition}
We say that $\big(\Omega,\mathcal{F},\mathbb{P},\left\{ \mathcal{F}%
_{s}\right\} _{s\geq t},W,X,K\big)$ is a weak solution of (\ref{reflected
SDE}) if $\big(\Omega,\mathcal{F},\mathbb{P},\left\{ \mathcal{F}_{s}\right\}
_{s\geq t}\big)$ is a stochastic basis, $W$ is a $d^{\prime}$-dimensional
Brownian motion with respect to this basis, $X$ is a continuous adapted
process and $K$ is a continuous bounded variation process such that $X_{s}\in%
\bar{G}$, $\forall s\in\left[ t,T\right] $ and system (\ref{reflected SDE})
is satisfied.
\end{definition}

The main results are the following two theorems. The first one consists in
establishing the weak uniqueness (in law) of the solution for (\ref%
{reflected SDE}) and the continuous dependence in law with respect to the
initial data.

\begin{theorem}
\label{Main result 2}Under the assumptions $\mathrm{(A}_{1}-\mathrm{A}_{2}%
\mathrm{)}$, there exists a unique weak solution $({X_{s}^{t,x},{K}%
_{s}^{t,x})}_{s\in\left[ t,T\right] }$ of SDE (\ref{reflected SDE}).
Moreover,%
\begin{equation*}
({X^{t,x,n},{K}^{t,x,n})}%
\xrightarrow[u]{\;\;\;\;*\;\; \;\;}%
({X^{t,x},{K}^{t,x})}
\end{equation*}
and the application%
\begin{equation*}
\left[ 0,T\right] \times\bar{G}\ni\left( t,x\right) \mapsto({X^{t,x},{K}%
^{t,x})}
\end{equation*}
is continuous in law.
\end{theorem}

Once this result for the forward part is established we then associate a
BSDE involving Stieltjes integral with respect to the increasing process $%
k^{t,x}$ in order to obtain the probabilistic representation for the
viscosity solution of PDE (\ref{PDE Neumann approx}).

The next result provides the approximation of a viscosity solution for
system (\ref{PDE Neumann}) by the solutions sequence of (\ref{PDE Neumann
approx}).

\begin{theorem}
\label{Main result 1}Under the assumptions $\mathrm{(A}_{1}-\mathrm{A}_{3}%
\mathrm{)}$, there exist continuous functions $u^{n}:[0,T]\times \mathbb{R}%
^{d}\rightarrow\mathbb{R}^{d}$ and $u:[0,T]\times\bar{G}\rightarrow\mathbb{R}%
^{d}$ such that $u^{n}$ is a viscosity solution for system (\ref{PDE Neumann
approx}), $u$ is a viscosity solution for system (\ref{PDE Neumann}) with
Neumann boundary conditions and, in addition,%
\begin{equation*}
\lim_{n\rightarrow\infty}u^{n}(t,x)=u(t,x),\;\forall\left( t,x\right)
\in\lbrack0,T]\times\bar{G}.
\end{equation*}
\end{theorem}

\section{Proof of Theorem \protect\ref{Main result 2}}

We shall divide the proof of this Theorem into several lemmas. First of all
we recall that the existence of a weak solution is given, under assumption $%
\mathrm{(A}_{1}\mathrm{)}$, by \cite[Theorem 3.2]{li-sz/84}.

For the simplicity of presentation we suppress from now on the explicit
dependence on $\left( t,x\right) $ in the notation of the solution of (\ref%
{reflected SDE}) and (\ref{penalized reflected SDE}).

We first prove an estimation result for the solutions of the penalized SDE (%
\ref{penalized reflected SDE}).

\begin{lemma}
\label{a-priori estimation}Under assumption $\mathrm{(A}_{1}\mathrm{)}$, for
any $q\geq1$, there exists a constant $C>0$, depending only on $d,T$ and $q$%
, such that%
\begin{equation}
\mathbb{E}\left( \sup\nolimits_{s\in\left[ t,T\right] }|{X_{s}^{n}|}%
^{2q}\right) +\mathbb{E}\left( \sup\nolimits_{s\in\left[ t,T\right] }|{%
K_{s}^{n}|}^{2q}\right) +\mathbb{E}|{K}^{n}|_{\left[ t,T\right] }^{q}\leq
C,\;\forall n\in\mathbb{N}\,.  \label{bound for approx}
\end{equation}
\end{lemma}

\begin{proof}
Without loss of generality we can assume that $0\in G$. From It\^{o}'s
formula applied for $|{X_{s}^{n}|}^{2}$ it can be deduced that%
\begin{equation*}
\begin{array}{r}
|{X_{s}^{n}|}^{2}+2{\displaystyle\int_{t}^{s}}\left\langle {X_{r}^{n}},d{K}%
_{r}^{n}\right\rangle =\left\vert x\right\vert ^{2}+2{\displaystyle%
\int_{t}^{s}}\left\langle {X_{r}^{n}},b\left( {X}_{r}^{n}\right)
\right\rangle dr+2{\displaystyle\int_{t}^{s}}\left\langle {X_{r}^{n}}%
,\sigma\left( {X}_{r}^{n}\right) {dW_{r}}\right\rangle \medskip \\
{+\displaystyle\int_{t}^{s}}|{\sigma({X_{r}^{n})|}}^{2}{dr,\;s\in}\left[ t,T%
\right] .%
\end{array}%
\end{equation*}
Write $\tau_{m}:=\inf\left\{ s\in\left[ t,T\right] :|{X_{s}^{n}|}\geq
m\right\} \wedge T$, $m\in\mathbb{N}^{\ast}$ and by the above,%
\begin{equation*}
\begin{array}{r}
|{X_{s\wedge\tau_{m}}^{n}|}^{2}+2{\displaystyle\int_{t}^{s\wedge\tau_{m}}}%
\left\langle {X_{r}^{n}},d{K}_{r}^{n}\right\rangle \leq C+\left\vert
x\right\vert ^{2}+C{\displaystyle\int_{t}^{s\wedge\tau_{m}}}\left\vert {%
X_{r}^{n}}\right\vert dr+2{\displaystyle\int_{t}^{s\wedge\tau_{m}}}%
\left\langle {X_{r}^{n}},\sigma\left( {X}_{r}^{n}\right) {dW_{r}}%
\right\rangle ,\medskip \\
{s\in}\left[ t,T\right] .%
\end{array}%
\end{equation*}
Here and in what follows $C>0$ will denote a generic constant which is
allowed to vary from line to line.

Therefore%
\begin{equation*}
\begin{array}{r}
\left( |{X_{s\wedge \tau _{m}}^{n}|}^{2}+{\displaystyle\int_{t}^{s\wedge
\tau _{m}}}\left\langle {X_{r}^{n}},d{K}_{r}^{n}\right\rangle \right)
^{q}\leq C\left( 1+\left\vert x\right\vert ^{2q}\right) +C\left( {%
\displaystyle\int_{t}^{s\wedge \tau _{m}}}\left\vert {X_{r}^{n}}\right\vert
^{2}dr\right) ^{q}\medskip \\
+C\left\vert {\displaystyle\int_{t}^{s\wedge \tau _{m}}}\left\langle {%
X_{r}^{n}},\sigma \left( {X}_{r}^{n}\right) {dW_{r}}\right\rangle
\right\vert ^{q},%
\end{array}%
\end{equation*}%
and%
\begin{equation}
\begin{array}{l}
\mathbb{E}\sup_{r\in \left[ t,s\right] }\left( |{X_{r\wedge \tau _{m}}^{n}|}%
^{2}+{\displaystyle\int_{t}^{r\wedge \tau _{m}}}\left\langle {X_{u}^{n}},d{K}%
_{u}^{n}\right\rangle \right) ^{q}\leq C\,\left( 1+\left\vert x\right\vert
^{2q}\right) \medskip \\
\quad +C\,\mathbb{E}{\displaystyle\int_{t}^{s\wedge \tau _{m}}}\sup_{u\in %
\left[ t,r\right] }\left\vert {X_{u}^{n}}\right\vert ^{2q}dr\medskip +C\,%
\mathbb{E}\sup_{r\in \left[ t,s\right] }\left\vert {\displaystyle%
\int_{t}^{r\wedge \tau _{m}}}\left\langle {X_{u}^{n}},\sigma \left( {X}%
_{u}^{n}\right) {dW_{u}}\right\rangle \right\vert ^{q}.%
\end{array}
\label{energy ineq}
\end{equation}%
By Burkholder-Davis-Gundy inequality we deduce%
\begin{equation*}
\begin{array}{r}
\mathbb{E}\sup_{r\in \left[ t,s\right] }\left\vert {\displaystyle%
\int_{t}^{r\wedge \tau _{m}}}\left\langle {X_{u}^{n}},\sigma \left( {X}%
_{u}^{n}\right) {dW_{u}}\right\rangle \right\vert ^{q}\leq C\,\mathbb{E}%
\left\vert {\displaystyle\int_{t}^{s\wedge \tau _{m}}}\left\vert {X_{u}^{n}}%
\right\vert ^{2}\left\vert \sigma \left( {X}_{u}^{n}\right) \right\vert ^{2}{%
du}\right\vert ^{q/2}\medskip \\
\leq C\,\mathbb{E}\left\vert {\displaystyle\int_{t}^{s\wedge \tau _{m}}}%
\left\vert {X_{u}^{n}}\right\vert ^{2}{du}\right\vert ^{q/2}\leq C\,\left( 1+%
\mathbb{E}{\displaystyle\int_{t}^{s\wedge \tau _{m}}}\sup_{u\in \left[ t,r%
\right] }\left\vert {X_{u}^{n}}\right\vert ^{2q}{dr}\right) ,%
\end{array}%
\end{equation*}%
and (\ref{energy ineq}) yields%
\begin{equation*}
\mathbb{E}\sup\nolimits_{r\in \left[ t,s\wedge \tau _{m}\right] }|{X_{r}^{n}|%
}^{2q}\leq C\,\left( 1+\left\vert x\right\vert ^{2q}+\mathbb{E}{\displaystyle%
\int_{t}^{s}}\sup\nolimits_{u\in \left[ t,r\wedge \tau _{m}\right]
}\left\vert {X_{u}^{n}}\right\vert ^{2q}dr\right) ,\;\forall s\in \left[ t,T%
\right] ,
\end{equation*}%
since from (\ref{subdif ineq}) applied for $z=0\in G$, we have%
\begin{equation*}
{\int_{t}^{s}}\left\langle {X_{r}^{n}},d{K}_{r}^{n}\right\rangle =n{%
\int_{t}^{s}}\left\langle {X_{r}^{n}},{\delta }({X_{r}^{n})}\right\rangle
dr\geq 0.
\end{equation*}%
From the Gronwall lemma,%
\begin{equation*}
\mathbb{E}\sup\nolimits_{r\in \left[ t,s\wedge \tau _{m}\right] }|{X_{r}^{n}|%
}^{2q}\leq C\,\left( 1+\left\vert x\right\vert ^{2q}\right) ,\;\forall
\,n\in \mathbb{N}\,.
\end{equation*}%
Taking $m\rightarrow \mathbb{\infty }$ it follows that%
\begin{equation}
\mathbb{E}\sup\nolimits_{r\in \left[ t,T\right] }|{X_{r}^{n}|}^{2q}\leq
C,\;\forall \,n\in \mathbb{N}\,.  \label{bound for X^n}
\end{equation}%
Once again from (\ref{energy ineq}) and (\ref{bound for X^n}) we obtain%
\begin{equation*}
\mathbb{E}\left( {\int_{t}^{T}}\left\langle {X_{r}^{n}},d{K}%
_{r}^{n}\right\rangle \right) ^{q}\leq C\,\left( 1+\left\vert x\right\vert
^{2q}\right) .
\end{equation*}%
We have that there exists $\varepsilon >0$ such that the ball $\bar{B}\left(
0,\varepsilon \right) \subset G$, and, for $z=\varepsilon \frac{{K}_{s}^{n}-{%
K}_{t}^{n}}{\left\vert {K}_{s}^{n}-{K}_{t}^{n}\right\vert }\in G$,
inequality (\ref{subdif ineq}) becomes%
\begin{equation*}
\varepsilon \left\vert {K}_{v}^{n}-{K}_{u}^{n}\right\vert \leq {\int_{u}^{v}}%
\left\langle {X_{r}^{n}},d{K}_{r}^{n}\right\rangle ,\;\forall t\leq u\leq
v\leq T
\end{equation*}%
and by the definition of total variation of ${K}^{n},$ it follows that%
\begin{equation*}
\varepsilon ^{q}\mathbb{E}\left( \left\vert {K}^{n}\right\vert _{\left[ t,T%
\right] }^{q}\right) \leq \mathbb{E}\Big({\int_{t}^{T}}\left\langle {%
X_{r}^{n}},d{K}_{r}^{n}\right\rangle \Big)^{q}\leq C~.
\end{equation*}%
\hfill
\end{proof}

\begin{lemma}
Under assumption $\mathrm{(A}_{1}\mathrm{)}$ the sequence $({X_{s}^{n},{K}%
_{s}^{n},{k}_{s}^{n})}_{s\in\left[ t,T\right] }$ is tight with respect to
the $\mathrm{S}$-topology.
\end{lemma}

\begin{proof}
In order to obtain the $\mathrm{S}$-tightness of a sequence of integrable c%
\`{a}dl\`{a}g processes $U^{n}$, $n\geq 1$, we shall use the sufficient
condition given e.g. in \cite[Appendix A]{le/02} which consists in proving
the uniform boundedness for:%
\begin{equation*}
\mathrm{CV}_{T}\left( U^{n}\right) +\mathbb{E}{\Big(}\sup_{s\in \lbrack
t,T]}\left\vert U_{s}^{n}\right\vert {\Big)},
\end{equation*}%
where%
\begin{equation}
\mathrm{CV}_{T}(U^{n}):=\sup_{\pi }{\sum_{i=0}^{m-1}{\mathbb{E}}\Big[{\big|}{%
\mathbb{E}}[U_{t_{i+1}}^{n}-U_{t_{i}}^{n}/{\mathcal{F}}_{t_{i}}]{\big|}\Big]}
\label{def cond var}
\end{equation}%
defines the conditional variation of $U^{n}$, with the supremum taken over
all finite partitions $\pi :t=t_{0}<t_{1}<\cdots <t_{m}=T.$

Using Lemma \ref{a-priori estimation}, we deduce that there exists a
constant $C>0$ such that for every $n\in\mathbb{N}^{\ast}$%
\begin{equation*}
\mathrm{CV}_{T}\left( K^{n}\right) +\mathbb{E}{\Big(}\sup_{s\in\lbrack
t,T]}\left\vert K_{s}^{n}\right\vert {\Big)}\leq\mathbb{E}\left( \left\vert
K^{n}\right\vert _{[t,T]}\right) +\mathbb{E}{\Big(}\sup_{s\in\lbrack
t,T]}\left\vert K_{s}^{n}\right\vert {\Big)}\leq C.
\end{equation*}
Since $k^{n}$ is increasing and $l\in C_{b}^{2}\left( \mathbb{R}^{d}\right) $%
, then there exist constants $M,C>0$ such that for every $n\in\mathbb{N}%
^{\ast}$
\begin{align*}
\mathrm{CV}_{T}\left( k^{n}\right) +\mathbb{E}{\Big(}\sup\limits_{s\in
\lbrack t,T]}\left\vert k_{s}^{n}\right\vert {\Big)} & \leq2\mathbb{E}\left(
k_{T}^{n}\right) =2\mathbb{E}{\Big(\int_{t}^{T}\left\langle \nabla\ell({%
X_{r}^{n})},dK_{r}^{n}\right\rangle \Big)} \\
& \leq2\mathbb{E}{\Big(\sup\limits_{s\in\lbrack t,T]}\left\vert \nabla \ell({%
X_{s}^{n})}\right\vert \cdot\left\vert K^{n}\right\vert _{\left[ t,T\right] }%
\Big)} \\
& \leq2M\,\mathbb{E}\left( {\left\vert K^{n}\right\vert _{\left[ t,T\right] }%
}\right) \\
& \leq2M\,C.
\end{align*}
By the definition of $V^{n}$, assumption $\mathrm{(A}_{1}\mathrm{)}$ and the
fact the conditional variation of a martingale is $0$, we obtain for each $%
n\in\mathbb{N}^{\ast}$,
\begin{align*}
\mathrm{CV}_{T}\left( V^{n}\right) & \leq\mathrm{CV}_{T}\left( {%
\int_{t}^{\cdot}b({X_{r}^{t,x,n})}dr}\right) +\mathrm{CV}_{T}\left( {%
\int_{t}^{\cdot}\sigma({X_{r}^{t,x,n})}dW_{r}}\right) \\
& =\mathrm{CV}_{T}\left( {\int_{t}^{\cdot}b({X_{r}^{t,x,n})}dr}\right) \leq%
\mathbb{E}\left( {\int_{t}^{T}}\left\vert {b({X_{r}^{t,x,n})}}\right\vert {dr%
}\right) \\
& \leq\left( T-t\right) \,M\leq C
\end{align*}
Therefore (see also Lemma \ref{a-priori estimation}), there exists $C>0$
such that for every $n\in\mathbb{N}^{\ast}$
\begin{equation*}
\mathrm{CV}_{T}\left( X^{n}\right) +\mathbb{E}{\Big(}\sup_{s\in\lbrack
t,T]}\left\vert X_{s}^{n}\right\vert {\Big)}\leq\mathrm{CV}_{T}\left(
V^{n}\right) +\mathrm{CV}_{T}\left( K^{n}\right) +\mathbb{E}{\Big(}%
\sup_{s\in\lbrack t,T]}\left\vert X_{s}^{n}\right\vert {\Big)}\leq C.
\end{equation*}
\hfill
\end{proof}

\begin{lemma}
\label{Uniq of X}Under the assumptions $\mathrm{(A}_{1}\mathrm{-A}_{2}%
\mathrm{)}$, the uniqueness in law of the stochastic process $(X_{s})_{s\in%
\left[ t,T\right] }$ holds.
\end{lemma}

\begin{proof}
Let $\big(\Omega ,\mathcal{F},\mathbb{P},\left\{ \mathcal{F}_{t}\right\}
_{t\geq 0},W,X,K\big)$ be a weak solution of (\ref{reflected SDE}) and $f\in
C^{1,2}\left( \left[ 0,T\right] \times \bar{G}\right) $. We apply It\^{o}'s
formula to $f\left( s,X_{s}\right) $:%
\begin{equation}
\begin{array}{l}
\displaystyle f\left( s,X_{s}\right) =f\left( t,x\right) +{\int_{t}^{s}\Big(}%
\frac{\partial f}{\partial r}+Lf{\Big)}\left( r,X_{r}\right) dr-{\int_{t}^{s}%
}\left\langle \nabla _{x}f\left( r,X_{r}\right) ,\nabla \ell \left(
X_{r}\right) \right\rangle dk_{r}\medskip \\
\quad \quad \quad \quad \quad +{\displaystyle\int_{t}^{s}}\left\langle
\nabla _{x}f\left( r,X_{r}\right) ,\sigma \left( X_{r}\right)
dW_{r}\right\rangle .%
\end{array}
\label{Ito uniqueness}
\end{equation}%
Since $\sigma \sigma ^{\ast }$ is supposed to be invertible, we deduce,
using Krylov's inequality for the reflecting diffusions (see \cite[Theorem
5.1]{ro-sl/97}), that for any $s\in \left[ t,T\right] ,$%
\begin{equation*}
\begin{array}{l}
\mathbb{E}\displaystyle{\int_{t}^{s}\Big|\Big(}\frac{\partial f}{\partial r}%
+Lf{\Big)}\left( r,X_{r}\right) {\Big|}\mathbb{1}_{\left\{ X_{r}\in \partial
G\right\} }dr\medskip \\
\quad \leq C~\left( \displaystyle\int_{t}^{s}\int_{G}\det \left( \sigma
\sigma ^{\ast }\right) ^{-1}{\Big(}\frac{\partial f}{\partial r}+Lf{\Big)}%
^{d+1}\mathbb{1}_{\partial G}~{drdx}\right) ^{\frac{1}{d+1}}=0~{.}%
\end{array}%
\end{equation*}%
and equality (\ref{Ito uniqueness}) becomes%
\begin{equation*}
\begin{array}{l}
f\left( s,X_{s}\right) =f\left( t,x\right) +{\displaystyle\int_{t}^{s}\Big(}%
\frac{\partial f}{\partial r}+Lf{\Big)}\left( r,X_{r}\right) \mathbb{1}%
_{\left\{ X_{r}\in G\right\} }dr-{\displaystyle\int_{t}^{s}}\left\langle
\nabla _{x}f\left( r,X_{r}\right) ,\nabla \ell \left( X_{r}\right)
\right\rangle dk_{r}\medskip \\
\quad \quad \quad \quad +{\displaystyle\int_{t}^{s}}\left\langle \nabla
_{x}f\left( r,X_{r}\right) ,\sigma \left( X_{r}\right) dW_{r}\right\rangle
~,\;\mathbb{P}\text{-a.s.}%
\end{array}%
\end{equation*}%
Therefore
\begin{equation*}
f\left( s,X_{s}\right) -f\left( t,x\right) -{\displaystyle\int_{t}^{s}\Big(}%
\frac{\partial f}{\partial r}+Lf{\Big)}\left( r,X_{r}\right) \mathbb{1}%
_{\left\{ X_{r}\in G\right\} }dr
\end{equation*}%
is a $\mathbb{P}$-supermartingale whenever $f\in C^{1,2}\left( \left[ 0,T%
\right] \times \bar{G}\right) $ satisfies%
\begin{equation*}
\left\langle \nabla _{x}f\left( s,x\right) ,\nabla \ell \left( x\right)
\right\rangle \geq 0,\forall x\in \partial G.
\end{equation*}%
From \cite[Theorem 5.7]{st-va/71} (applied with $\phi =-\ell $, $\gamma
:=\nabla \phi $ and $\rho :=0$) we have that the solution to the
supermartingale problem is unique for each starting point $\left( t,x\right)
$, therefore our solution process $\left( X_{s}\right) _{s\in \left[ t,T%
\right] }$ is unique in law.\hfill
\end{proof}

\begin{remark}
\label{Uniq of couple}Following the remark of El Karoui \cite[Theorem 6]%
{ka/75} we obtain the uniqueness in law of the couple $\left( X,K\right) $,
since the increasing process $k$ depends only on the solution $X$ (and not
on the Brownian motion). The uniqueness is essential in order to formulate
the issue of the continuity with respect to the initial data.
\end{remark}

\begin{lemma}
\label{conv in C}We suppose that the assumptions $\mathrm{(A}_{1}-\mathrm{A}%
_{2}\mathrm{)}$ are satisfied. Then%
\begin{equation*}
\begin{array}{rl}
\left( i\right) & ({X^{n},{K}^{n})}%
\xrightarrow[u]{\;\;\;\;*\;\; \;\;}%
({X,{K})},\medskip \\
\left( ii\right) & {k^{n}}%
\xrightarrow[u]{\;\;\;*\; \;\;}%
k.%
\end{array}%
\end{equation*}
\end{lemma}

\begin{proof}
\newline
\noindent $\left( i\right) $ First we will prove the convergence:%
\begin{equation}
({X^{n},{K}^{n})}%
\xrightarrow[\mathrm{S}]{\;\;\;*\; \;\;}%
({X,{K})}.  \label{conv in S}
\end{equation}%
We shall apply \cite[Theorem 4.3 $\left( iii\right) $]{la-sl/12}. We recall
that we have the uniqueness of the weak solution. For any $n\in \mathbb{N}$,
$s\in \left[ t,T\right] $, let $H_{s}^{n}:=x\in \bar{G}$ and the processes $%
Z_{s}^{n}:=\left( s,W_{s}\right) $. Our equation can be written as%
\begin{equation*}
{X_{s}^{n}=H_{s}^{n}+\int_{t}^{s}{\left\langle \left( {b,\sigma }\right) {({%
X_{r}^{t,x,n})}},dZ_{s}^{n}\right\rangle -}}K_{s}^{n}\,,\;\forall s\in \left[
t,T\right] .
\end{equation*}%
The processes $Z^{n}$ satisfies the \textbf{(UT)} condition (introduced in
\cite{st/85}), since for any discrete predictable processes $U^{n},\bar{U}%
^{n}$ of the form $U_{s}^{n}:=U_{0}^{n}+\sum\nolimits_{i=0}^{k}U_{i}^{n}$,
respectively $\bar{U}_{s}^{n}:=\bar{U}_{0}^{n}+\sum\nolimits_{i=0}^{k}\bar{U}%
_{i}^{n}$ with $\left\vert U_{i}^{n}\right\vert ,\left\vert \bar{U}%
_{i}^{n}\right\vert \leq 1$,%
\begin{align*}
\mathbb{E}\left\vert \int_{0}^{q}U_{s}^{n}ds+\int_{0}^{q}\bar{U}%
_{s}^{n}dW_{s}\right\vert ^{2}& \leq 2\mathbb{E}\left\vert
\int_{0}^{q}U_{s}^{n}ds\right\vert ^{2}+2\mathbb{E}\left\vert \int_{0}^{q}%
\bar{U}_{s}^{n}dW_{s}\right\vert ^{2} \\
& \leq 2q^{2}+2\mathbb{E}\int_{0}^{q}\left\vert \bar{U}_{s}^{n}\right\vert
^{2}ds\leq 2q\left( q+1\right) .
\end{align*}%
Therefore the assumptions of \cite[Theorem 4.3]{la-sl/12} are satisfied and
thus we obtain that%
\begin{equation*}
{X^{n}%
\xrightarrow[\mathrm{S}]{\;\;\;*\; \;\;}%
X}.
\end{equation*}%
Using once again \cite[Theorem 4.3 $\left( ii\right) $]{la-sl/12} and
definition (\ref{def V^n}) we deduce that%
\begin{equation*}
\left( X_{t_{1}}^{n},X_{t_{2}}^{n},\ldots ,X_{t_{m}}^{n},{V^{n}}\right) {%
\xrightarrow[]{\;\;\;*\; \;\;}%
}\left( X_{t_{1}},X_{t_{2}},\ldots ,X_{t_{m}},{V}\right) ,
\end{equation*}%
for any partition $t=t_{0}<t_{1}<\cdots <t_{m}=T.$ The above convergence is
considered in law, on the space $\left( \mathbb{R}^{d}\right) ^{m}\times
\mathcal{D}([0,T]\,,\,{\mathbb{R}}^{d})$ endowed with the product between
the usual topology on $\left( \mathbb{R}^{d}\right) ^{m}$ and the Skorohod to%
{\footnotesize \-}po{\footnotesize \-}lo{\footnotesize \-}gy $J_{1}$.

Hence%
\begin{equation*}
\left( X^{n},{V^{n}}\right)
\xrightarrow[\mathrm{S}]{\;\;\;*\; \;\;}%
\left( X,{V}\right) ,
\end{equation*}
since $\left( X^{n},V^{n}\right) _{n}$ is tight.

It is known that the space $\mathcal{D}([0,T]\,,\,{\mathbb{R}}^{d})$ of c%
\`{a}dl\`{a}g functions endowed with $\mathrm{S}$-topology is not a linear
topological space, but the sequential continuity of the addition, with
respect to the $\mathrm{S}$-topology, is fulfilled (see Jakubowski \cite[%
Remark 3.12]{ja/97}). Therefore%
\begin{equation*}
K^{n}=V^{n}-X^{n}%
\xrightarrow[\mathrm{S}]{\;\;\;*\; \;\;}%
V-X=K.
\end{equation*}%
$\bigskip $In order to obtain the uniform convergence\footnote[4]{%
We are thankful to professor L. S\l omi\'{n}ski for his useful suggestion in
the proof of this part.} of the sequence $({X^{n},{K}^{n})}_{n}$ we remark
that, since ${V^{n}%
\xrightarrow[J_{1}]{\;\;\;*\; \;\;}%
V}$ and $V^{n}$, $V$ are continuous, this convergence is uniform in
distribution:%
\begin{equation*}
{V^{n}%
\xrightarrow[u]{\;\;\;*\; \;\;}%
V.}
\end{equation*}%
Using the Skorohod theorem, there exists a new probability space $\big(\hat{%
\Omega},\mathcal{\hat{F}},\mathbb{\hat{P}}\big)$ on which we can define
random variables $\hat{V},\hat{V}^{n}$ such that%
\begin{equation*}
\hat{V}\xlongequal{law}V,\;\;\hat{V}^{n}\xlongequal{law}V^{n},\;\forall n\in
\mathbb{N}
\end{equation*}%
and%
\begin{equation*}
{\sup_{s\in \left[ t,T\right] }}|{\hat{V}_{s}^{n}-\hat{V}}_{s}|%
\xrightarrow[\;\;\;\;\;\;\;]{a.s.}{0.}
\end{equation*}%
Let $\hat{X}^{n}$ be the solution of the equation%
\begin{equation*}
\hat{X}_{s}^{n}+{\int_{t}^{s}n\delta }({\hat{X}_{r}^{n})}dr=\hat{V}%
_{s}^{n}\,,\;s\in \left[ t,T\right] ,
\end{equation*}%
$\bar{X}^{n}$ be the solution of%
\begin{equation*}
\bar{X}_{s}^{n}+{\int_{t}^{s}n\delta }({\bar{X}_{r}^{n})}dr=\hat{V}%
_{s}\,,\;s\in \left[ t,T\right] ,
\end{equation*}%
and denote%
\begin{equation*}
\hat{K}_{s}^{n}:={\int_{t}^{s}n\delta }({\hat{X}_{r}^{n})}dr\text{,\ \ }\bar{%
K}_{s}^{n}:={\int_{t}^{s}n\delta }({\bar{X}_{r}^{n})}dr.
\end{equation*}%
It is easy to prove (see, e.g., \cite[Lemma 2.2]{la-sl/03} or \cite[Lemma 2.2%
]{ta/79}) that%
\begin{equation*}
{\big|}\hat{X}_{s}^{n}-\bar{X}_{s}^{n}{\big|}^{2}\leq {\big|}\hat{V}_{s}^{n}-%
\hat{V}_{s}{\big|}^{2}+2\int_{t}^{s}\big\langle(\hat{V}_{s}^{n}-\hat{V}%
_{s})-(\hat{V}_{r}^{n}-\hat{V}_{r}),d(\hat{K}_{r}^{n}-\bar{K}_{r}^{n})%
\big\rangle,
\end{equation*}%
therefore%
\begin{equation}
\sup_{s\in \left[ t,T\right] }{\big|}\hat{X}_{s}^{n}-\bar{X}_{s}^{n}{\big|}%
^{2}\leq \sup_{s\in \left[ t,T\right] }{\big|\hat{V}_{s}^{n}-\hat{V}_{s}\big|%
}^{2}+4\sup_{s\in \left[ t,T\right] }{\big|\hat{V}_{s}^{n}-\hat{V}_{s}\big|}%
\big(|\hat{K}^{n}|_{\left[ t,T\right] }+|\bar{K}^{n}|_{\left[ t,T\right] }%
\big).  \label{differ for X}
\end{equation}%
Since $(\hat{X}^{n},\hat{K}^{n})\xlongequal[]{law}\left( X^{n},K^{n}\right) $
and $|K^{n}|_{\left[ t,T\right] }$ is bounded in probability by inequality (%
\ref{bound for approx}), $|\hat{K}^{n}|_{\left[ t,T\right] }$ is bounded in
probability. Applying \cite[Theorem 2.7]{la-sl/03}, it follows that $|\bar{K}%
^{n}|_{\left[ t,T\right] }$ is also bounded in probability.

But%
\begin{equation*}
\sup_{s\in \left[ t,T\right] }{\big|\hat{V}_{s}^{n}-\hat{V}_{s}\big|%
\xrightarrow[]{prob}0,}
\end{equation*}%
therefore, from (\ref{differ for X}),%
\begin{equation}
\sup_{s\in \left[ t,T\right] }{\big|}\hat{X}_{s}^{n}-\bar{X}_{s}^{n}{\big|}%
^{2}{\xrightarrow[\;\;\;\;\;\;\;\;\;\;]{prob}0.}  \label{conv to 0}
\end{equation}%
On the other hand, let $\hat{X}$ be the solution of the Skorohod problem%
\begin{equation*}
\hat{X}_{s}+\hat{K}_{s}=\hat{V}_{s}\,,\;s\in \left[ t,T\right] .
\end{equation*}%
It can be shown (see the proof of \cite[Theorem 2.1]{li-sz/84} or the proof
of \cite[Theorem 4.17]{pa-ra/12}) that%
\begin{equation*}
\sup_{s\in \left[ t,T\right] }{%
\big|%
}\bar{X}_{s}^{n}-\hat{X}_{s}{%
\big|%
}^{2}\;{\xrightarrow[]{prob}0,}
\end{equation*}%
therefore, from (\ref{conv to 0}),%
\begin{equation*}
\sup_{s\in \left[ t,T\right] }{%
\big|%
}\hat{X}_{s}^{n}-\hat{X}_{s}{%
\big|%
}^{2}\;{\xrightarrow[]{prob}0.}
\end{equation*}%
Since $\hat{K}_{s}=\hat{V}_{s}-\hat{X}_{s}\,$, $\hat{K}_{s}^{n}=\hat{V}%
_{s}^{n}-\hat{X}_{s}^{n}\,,\;s\in \left[ t,T\right] ,$%
\begin{equation*}
(\hat{X}^{n},\hat{K}^{n}){\xrightarrow[u]{prob}(}\hat{X},\hat{K}),
\end{equation*}%
and%
\begin{equation*}
(\hat{X}^{n},\hat{K}^{n})\xlongequal[]{law}\left( X^{n},K^{n}\right) ,
\end{equation*}%
the conclusion follows.

\noindent $\left( ii\right) $ In order to pass to the limit in the integral $%
{\displaystyle{\int_{t}^{s}}\left\langle \nabla \ell ({X_{r})}%
,dK_{r}\right\rangle }$, we apply the stochastic version of Helly-Bray
theorem given by \cite[Proposition 3.4]{za/08}. For the convenience of the
reader we give the statement of that result:

\begin{lemma}
Let $\left( X^{n},K^{n}\right) :\left( \Omega^{n},\mathcal{F}^{n},\mathbb{P}%
^{n}\right) \longrightarrow\mathcal{C}\left( \left[ 0,T\right] ,\mathbb{R}%
^{d}\right) $ be a sequence of random variables and $\left( X,K\right) $
such that%
\begin{equation*}
\left( X^{n},K^{n}\right)
\xrightarrow[u]{\;\;\;*\; \;\;}%
\left( X,K\right) .
\end{equation*}
If $\left( K^{n}\right) _{n}$ has bounded variation a.s. and%
\begin{equation*}
\sup_{n\in\mathbb{N}^{\ast}}\mathbb{P}\left( \left\vert K^{n}\right\vert _{%
\left[ 0,T\right] }>a\right) \longrightarrow0,\;\text{as }%
a\longrightarrow\infty,
\end{equation*}
then $K$ has a.s. bounded variation and
\begin{equation*}
\int_{0}^{T}\left\langle {X_{r}^{n}},dK_{r}^{n}\right\rangle
\xrightarrow[u]{\;\;\;*\; \;\;}%
\int_{0}^{T}\left\langle {X_{r}},dK_{r}\right\rangle ,\text{ as }%
n\longrightarrow\infty.
\end{equation*}
\end{lemma}

Ret{urning to the proof of Lemma \ref{conv in C}, the conclusion $\left(
ii\right) $ follows now easily,} since $k$ and $k^{n}$ are defined by (\ref%
{def k}) and (\ref{def k^n}) respectively.\hfill
\end{proof}

\begin{remark}
Let the assumptions $\mathrm{(A}_{1}\mathrm{-A}_{2}\mathrm{)}$ be satisfied.
Then the weak solution $(X_{s}^{t,x})_{s\in \left[ t,T\right] }$ is a strong
Markov process. Indeed, taking into account the equivalence between the
existence for the (sub-)mar{\footnotesize \-}tingale problem and the
existence of a weak solution for reflected SDE (\ref{reflected SDE}) (see
\cite[Theorem 7]{ka/75}), we obtain that the weak solution $%
(X_{s}^{t,x})_{s\in \left[ t,T\right] }$ is a strong Markov process since
the uniqueness holds (see \cite[Theorem 10]{ka/75}). In our situation, this
equivalence can be obtained by using Krylov's inequality for reflecting
diffusions.
$\medskip $
\end{remark}

The following result will finalize the proof of Theorem \ref{Main result 2}.

We extend the solution process to $\left[ 0,T\right] $ by denoting%
\begin{equation}
X_{s}^{t,x}:=x,\;K_{s}^{t,x}:=0,\;\forall s\in\lbrack0,t).
\label{extens X,K}
\end{equation}

\begin{lemma}
We suppose that the assumptions $\mathrm{(A}_{1}\mathrm{-A}_{2}\mathrm{)}$
are satisfied and let $(X_{s}^{t,x},K_{s}^{t,x})_{s\in\left[ t,T\right] }$
be the weak solution of (\ref{reflected SDE}). Then

\noindent$\left( i\right) $ the family $(X_{s}^{t,x},K_{s}^{t,x})_{s\in\left[
0,T\right] }$ is tight with respect to the initial data $\left( t,x\right) $%
, as family of $\mathcal{C}(\left[ 0,T\right] $,$\mathbb{R}^{d}\times\mathbb{%
R}^{d})$-valued random variables and

\noindent$\left( ii\right) $ the weak solution $(X_{s}^{t,x},K_{s}^{t,x})_{s%
\in\left[ t,T\right] }$ is continuous in law with respect to the initial
data $\left( t,x\right) $.
\end{lemma}

\begin{proof}
$\left( i\right) $ First let $\left( t,x\right) \in\left[ 0,T\right] \times%
\bar{G}$ be fixed. Denote as before $\left( X_{s},K_{s}\right)
:=(X_{s}^{t,x},K_{s}^{t,x})$.

Applying It\^{o}'s formula for the process $X_{s}-X_{r}$, where $r$ is fixed
and $s\geq r$, we deduce%
\begin{align*}
\left\vert X_{s}-X_{r}\right\vert ^{2}& =2{\displaystyle\int_{r}^{s}}%
\left\langle X_{u}-X_{r},b\left( X_{u}\right) \right\rangle du-2{%
\displaystyle\int_{r}^{s}}\left\langle X_{u}-X_{r},dK_{u}\right\rangle +{%
\displaystyle\int_{r}^{s}}\left\vert \sigma \left( X_{u}\right) \right\vert
^{2}du \\
& +2{\displaystyle\int_{r}^{s}}\left\langle X_{u}-X_{r},\sigma \left(
X_{u}\right) dW_{u}\right\rangle \\
& \leq 2{\displaystyle\int_{r}^{s}}\left\langle X_{u}-X_{r},b\left(
X_{u}\right) \right\rangle du+{\displaystyle\int_{r}^{s}}\left\vert \sigma
\left( X_{u}\right) \right\vert ^{2}du+2{\displaystyle\int_{r}^{s}}%
\left\langle X_{u}-X_{r},\sigma \left( X_{u}\right) dW_{u}\right\rangle ,
\end{align*}%
since $X_{u},X_{r}\in \bar{G}$ and%
\begin{equation*}
{\int_{r}^{s}}\left\langle z-X_{u},dK_{u}\right\rangle ={\int_{r}^{s}}%
\left\langle z-X_{u},\nabla \ell \left( X_{u}\right) \right\rangle
dk_{u}\leq 0,\;\forall \,0\leq r\leq s,\;\forall \,z\in \mathbb{R}^{d}.
\end{equation*}%
Therefore, using that $b,\sigma $ are bounded functions and $\bar{G}$ is a
bounded domain,%
\begin{equation}
\begin{array}{l}
\mathbb{E}{\Big(}\left\vert X_{s}-X_{r}\right\vert ^{8}{\Big)}\leq
C\left\vert s-r\right\vert ^{4}+C\mathbb{E}{\Big(}\sup_{v\in \left[ r,s%
\right] }{\displaystyle\int_{r}^{v}}\left\langle X_{u}-X_{r},\sigma \left(
X_{u}\right) dW_{u}\right\rangle {\Big)}^{4}\medskip \\
\quad \quad \quad \quad \quad \quad \quad \;\leq C\left\vert s-r\right\vert
^{4}+C\mathbb{E}{\Big(\displaystyle\int_{r}^{s}}\left\vert
X_{u}-X_{r}\right\vert ^{2}\left\vert \sigma \left( X_{u}\right) \right\vert
^{2}du{\Big)}^{2}\medskip \\
\quad \quad \quad \quad \quad \quad \quad \;\leq C\left\vert s-r\right\vert
^{4}+C\left\vert s-r\right\vert ^{2}\leq C\max \left( \left\vert
s-r\right\vert ^{4},\left\vert s-r\right\vert ^{2}\right) .%
\end{array}
\label{tight crit 1}
\end{equation}%
Concerning $K$, we remark first that%
\begin{equation*}
K_{s}-K_{r}={\int_{r}^{s}}b\left( X_{u}\right) du+{\int_{r}^{s}}\sigma
\left( X_{u}\right) dW_{u}-\left( X_{s}-X_{r}\right) .
\end{equation*}%
Hence%
\begin{equation}
\begin{array}{l}
\mathbb{E}{\Big(}\left\vert K_{s}-K_{r}\right\vert ^{8}{\Big)}\leq C\mathbb{E%
}{\Big(}\left\vert X_{s}-X_{r}\right\vert ^{8}{\Big)}+C\mathbb{E}{\Big(%
\displaystyle\int_{r}^{s}}b\left( X_{u}\right) du{\Big)}^{8}\medskip \\
\quad \quad \quad \quad \quad \quad \quad \quad \quad +C\mathbb{E}{\Big(}%
\sup_{v\in \left[ r,s\right] }{\displaystyle\int_{r}^{v}}\sigma \left(
X_{u}\right) dW_{u}{\Big)}^{8}\medskip \\
\quad \quad \quad \quad \quad \quad \quad \;\leq C\max \left( \left\vert
s-r\right\vert ^{4},\left\vert s-r\right\vert ^{2}\right) +C\left\vert
s-r\right\vert ^{8}+C\mathbb{E}{\Big(\displaystyle\int_{r}^{s}}\left\vert
\sigma \left( X_{u}\right) \right\vert ^{2}du{\Big)}^{4}\medskip \\
\quad \quad \quad \quad \quad \quad \quad \;\leq C\max \left( \left\vert
s-r\right\vert ^{8},\left\vert s-r\right\vert ^{2}\right) .%
\end{array}
\label{tight crit 2}
\end{equation}%
{Observe} that the constants in the right hand of the inequalities (\ref%
{tight crit 1}) and (\ref{tight crit 2}) do not depend on $\left( t,x\right)
$.

Therefore, applying a tightness criterion (see, e.g. \cite[Cap. I]{pa-ra/12}%
) the desired conclusion follows.$\bigskip $

$\left( ii\right) $ Taking into account the conclusion $\left( i\right) $
and the Prokhorov theorem, we have that if $\left( t_{n},x_{n}\right)
\rightarrow\left( t,x\right) $, as $n\rightarrow\mathbb{\infty}$, then there
exists a subsequence, still denoted by $\left( t_{n},x_{n}\right) $, such
that%
\begin{equation*}
X^{n}:=X^{t_{n},x_{n}}%
\xrightarrow[u]{\;\;\;*\; \;\;}%
X,\quad K^{n}:=K^{t_{n},x_{n}}%
\xrightarrow[u]{\;\;\;*\; \;\;}%
K,\;\text{as }n\rightarrow\mathbb{\infty}.
\end{equation*}
It remains to identify the limits, i.e. $X%
\xlongequal[]{law}%
X^{t,x}$ and $K%
\xlongequal[]{law}%
K^{t,x}$.

Since $\left( X^{n},K^{n},W^{n}\right) _{n}$ is a $\mathcal{C}(\left[ 0,T%
\right] ,\mathbb{R}^{d}\times \mathbb{R}^{d}\times \mathbb{R}^{d^{\prime }})$
tight sequence, by the Skorohod Theorem, we can choose a probability space $%
\big(\hat{\Omega},\mathcal{\hat{F}},\mathbb{\hat{P}}\big)$ (which can be
taken in fact as $\left( \left[ 0,1\right] ,\mathcal{B}_{\left[ 0,1\right]
},\mu \right) $ where $\mu $ is the Lebesgue measure), and $(\hat{X}^{n},%
\hat{K}^{n},\hat{W}^{n})$, $(\hat{X},\hat{K},\hat{W})$ defined on this
probability space, such that
\begin{equation*}
(\hat{X}^{n},\hat{K}^{n},\hat{W}^{n})%
\xlongequal[]{law}%
(X^{n},K^{n},W^{n}),\;\;(\hat{X},\hat{K},\hat{W})%
\xlongequal[]{law}%
(X,K,W)
\end{equation*}%
and%
\begin{equation*}
(\hat{X}^{n},\hat{K}^{n},\hat{W}^{n})\xrightarrow[]{a.s.}(\hat{X},\hat{K},%
\hat{W}),\;\text{as }n\rightarrow \mathbb{\infty }.
\end{equation*}%
Then, using \cite[Proposition 2.15]{pa-ra/12}, we deduce that $\hat{W}^{n}$
is an $\mathcal{F}_{t}^{\hat{W}^{n},\hat{X}^{n}}$-Brownian motion, $\hat{W}$
is an $\mathcal{F}_{t}^{\hat{W},\hat{X}}$-Brownian motion and (together with
the Lebesgue theorem) we infer that, for all $q\geq 1$,%
\begin{equation*}
\mathbb{E}{\big(}\sup_{s\in \left[ t,T\right] }{\big|}\hat{V}_{s}^{n}-\hat{V}%
_{s}{\big|}^{q}{\big)}\rightarrow 0,\;\text{as }n\rightarrow \mathbb{\infty }%
,
\end{equation*}%
where%
\begin{align*}
\hat{V}_{s}^{n}& :=x+\int_{t}^{s}b(\hat{X}_{r}^{n})dr+\int_{t}^{s}\sigma (%
\hat{X}_{r}^{n})d\hat{W}_{r}^{n}\text{\ \ and} \\
\hat{V}_{s}& :=x+\int_{t}^{s}b(\hat{X}_{r})dr+\int_{t}^{s}\sigma (\hat{X}%
_{r})d\hat{W}_{r}\,,\;s\in \left[ t,T\right] .
\end{align*}%
If $V^{n}$ is defined by%
\begin{equation*}
V_{s}^{n}:=x+\int_{t}^{s}b(X_{r}^{n})dr+\int_{t}^{s}\sigma
(X_{r}^{n})dW_{r}^{n}
\end{equation*}%
then $X_{s}^{n}+K_{s}^{n}=V_{s}^{n},\;\mathbb{P}$-a.s.; using \cite[%
Corollary 2.14]{pa-ra/12}, we see that%
\begin{equation*}
\left( X^{n},K^{n},W^{n},V^{n}\right)
\xlongequal[]{law}%
(\hat{X}^{n},\hat{K}^{n},\hat{W}^{n},\hat{V}^{n})\;\text{on }\mathcal{C}(%
\left[ 0,T\right] ,\mathbb{R}^{d}\times \mathbb{R}^{d}\times \mathbb{R}%
^{d^{\prime }}\times \mathbb{R}^{d})
\end{equation*}%
and%
\begin{equation*}
\hat{X}_{s}^{n}+\hat{K}_{s}^{n}=\hat{V}_{s}^{n},\;\;\text{a.s.}
\end{equation*}%
which yields, passing to the limit, that%
\begin{equation*}
\hat{X}_{s}+\hat{K}_{s}=\hat{V}_{s},\;\text{a.s.}
\end{equation*}%
Then the coupled process $(\hat{X}_{s},\hat{K}_{s})$ is a solution of (\ref%
{reflected SDE}) corresponding to the initial data $\left( t,x\right) $.
Taking into account the uniqueness in law of the solution $%
(X_{s}^{t,x},K_{s}^{t,x})_{s\in \left[ t,T\right] }\,$(see Remark \ref{Uniq
of couple}) we deduce that the whole sequence $(X_{s}^{n},K_{s}^{n})_{s\in %
\left[ t,T\right] }$ converges to the process $(X_{s}^{t,x},K_{s}^{t,x})_{s%
\in \left[ t,T\right] }$, and therefore the continuity with respect to $%
\left( t,x\right) $ follows.\hfill
\end{proof}

\section{BSDEs and nonlinear Neumann boundary problem}

Let us now consider the processes $(X_{s}^{t,x,n},k_{s}^{t,x,n})_{t\leq
s\leq T}$ and $(X_{s}^{t,x},k_{s}^{t,x})_{t\leq s\leq T}$ given by relations
(\ref{reflected SDE}) - (\ref{def k^n}), for $\left( t,x\right) \in\left[ 0,T%
\right] \times\bar{G}$.

In order to give the proof of Theorem \ref{Main result 1} we associate the
following generalized backward stochastic differential equations (BSDE for
short) on $\left[ t,T\right] $:%
\begin{equation}
{Y_{s}^{t,x}=g(X}_{T}^{t,x}){+\int_{s}^{T}f(r,{X_{r}^{t,x},Y{_{r}^{t,x}})}%
dr-\int_{s}^{T}U{_{r}^{t,x}}dM_{r}^{X^{t,x}}-\int_{s}^{T}h(r,{X_{r}^{t,x},Y{%
_{r}^{t,x}})dk_{r}^{t,x}}}\,,  \label{generalized BSDE}
\end{equation}
and respectively the BSDE corresponding to the solution of (\ref{penalized
reflected SDE})%
\begin{equation}
\begin{array}{r}
{Y_{s}^{t,x,n}=g(X}_{T}^{t,x,n}){+\displaystyle\int_{s}^{T}f(r,{%
X_{r}^{t,x,n},Y{_{r}^{t,x,n}})}dr-\displaystyle\int_{s}^{T}U{_{r}^{t,x,n}}%
dM_{r}^{X^{t,x,n}}}\medskip \\
{-\displaystyle\int_{s}^{T}h(r,{X_{r}^{t,x,n},Y{_{r}^{t,x,n}})dk_{r}^{t,x,n}}%
}\,,%
\end{array}
\label{penalized generalized BSDE}
\end{equation}
where%
\begin{equation}
{M_{s}^{X^{t,x}}:=\int_{t}^{s}\sigma({X_{r}^{t,x})}dW_{r}}\,,\text{\quad }{%
M_{s}^{X^{t,x,n}}:=\int_{t}^{s}\sigma({X_{r}^{t,x,n})}dW_{r}}
\label{def M, M^n}
\end{equation}
are the martingale part of the reflected diffusion process $X^{t,x}$ and $%
X^{t,x,n}$ respectively. We assume for simplicity that the processes $({{%
X_{s}^{t,x,n},K_{s}^{t,x,n})}}_{s\in\left[ t,T\right] }$ and $({{%
X_{s}^{t,x},K_{s}^{t,x})}}_{s\in\left[ t,T\right] }$ are considered on the
canonical space.

We recall that the coefficients $f,g$ and $h$ satisfy assumption \textrm{(A}$%
_{3}$\textrm{)}. Then, given the processes $({{X_{s}^{t,x,n},k_{s}^{t,x,n})}}%
_{s\in\left[ t,T\right] }$ and $({{X_{s}^{t,x},k_{s}^{t,x})}}_{s\in\left[ t,T%
\right] }$, this assumption ensures (see \cite{pa-zh/98}) the existence and
the uniqueness for the couples $({{Y_{s}^{t,x,n},U_{s}^{t,x,n})}}_{s\in\left[
t,T\right] }$ and $({{Y_{s}^{t,x},U_{s}^{t,x})}}_{s\in\left[ t,T\right] }$
respectively. Arguing as in \cite{bo-ca/04}, one can establish the following
result.

\begin{proposition}
\label{result_Boufoussi}Let the assumptions $\mathrm{(A}_{1}-\mathrm{A}_{3}%
\mathrm{)}$ be satisfied. Let $(Y_{s}^{t,x,n},U_{s}^{t,x,n})_{s\in\left[ t,T\right] }$ and $%
(Y_{s}^{t,x},U_{s}^{t,x})_{s\in\left[ t,T\right] }$ be the solutions of the
BSDEs (\ref{penalized generalized BSDE}) and (\ref{generalized BSDE}),
respectively. Then%
\begin{equation*}
\left( Y^{t,x,n},M^{t,x,n},H^{t,x,n}\right) {\xrightarrow[\mathrm{S}\mathbb{%
\times}\mathrm{S}\mathbb{\times}\mathrm{S}]{\;\;\;\;\;\;\;\;*\;\;\;\;\;\;\;}}%
\left( Y^{t,x},M^{t,x},H^{t,x}\right) ,
\end{equation*}
where%
\begin{equation*}
M_{s}^{t,x,n}:={\int_{t}^{s}U{_{r}^{t,x,n}}dM_{r}^{X^{t,x,n}}{\,,\;\;}%
H_{s}^{t,x,n}:={\int_{0}^{s}h(r,{X_{r}^{t,x,n},Y{_{r}^{t,x,n}})dk_{r}^{t,x,n}%
}},}
\end{equation*}%
\begin{equation}
{M_{s}^{t,x}:=\int_{t}^{s}U{_{r}^{t,x}}dM_{r}^{X^{t,x}}{\,,\;\;}H_{s}^{t,x}:=%
{\int_{0}^{s}}h(r,{X_{r}^{t,x},Y{_{r}^{t,x}})dk_{r}^{t,x}}}
\label{def M and H}
\end{equation}
and ${M^{X^{t,x,n}}}$ and ${M^{X^{t,x}}}$ are defined by (\ref{def M, M^n}).$%
\medskip$

Moreover, we have that $\lim\limits_{n\rightarrow\mathbb{\infty}%
}Y_{t}^{t,x,n}=Y_{t}^{t,x}\,.$
\end{proposition}

\begin{remark}
\label{Uniq of Y}The solution process $(Y_{s}^{t,x})_{s\in \left[ t,T\right]
}$ is unique in law. Indeed, following \cite[Theorem 3.4]{ka/97}, it can be
proven that, since the coefficients $b$ and $\sigma $ satisfy the
assumptions $\mathrm{(A}_{1}\mathrm{-A}_{2}\mathrm{)}$ and the solution
process has the Markov property, there exists a deterministic measurable
function $u$ such that the solution $Y_{s}^{t,x}=u(s,X_{s}^{t,x})$, $s\in %
\left[ t,T\right] $ $d\mathbb{P}\otimes ds$ a.s. The conclusion follows by
Proposition \ref{Uniq of X}{\ and the uniqueness (as a strong solution) of $%
Y $}.
\end{remark}

In the following, we extend $X^{t,x},K^{t,x}$ to $\left[ 0,T\right] $ as in (%
\ref{extens X,K}) and $(Y^{t,x},U^{t,x})$ by denoting%
\begin{equation*}
Y_{s}^{t,x}:=Y_{t}^{t,x},\;U_{s}^{t,x}:=0\text{\quad and\quad}{%
M_{s}^{X^{t,x}}:=0}\text{, }\forall s\in\lbrack0,t).
\end{equation*}

\begin{proposition}
\label{cont of Y}Let $\left( t_{n},x_{n}\right) \rightarrow\left( t,x\right)
$, as $n\rightarrow\mathbb{\infty}$. Then there exists a subsequence $%
(t_{n_{k}},x_{n_{k}})_{k\in\mathbb{N}}$ such that $Y^{t_{n_{k}},x_{n_{k}}}%
\xrightarrow[\mathrm{S}]{\;\;\;*\; \;\;}%
Y^{t,x}$.
\end{proposition}

\begin{proof}
The proof will follow the techniques used in \cite[Theorem 3.1]{bo-ca/04}
(see also \cite[Theorem 6.1]{pa/99}).

It is clear that%
\begin{equation}
\begin{array}{r}
Y_{s}^{t_{n},x_{n}}{=g(X}_{T}^{t_{n},x_{n}}){+\displaystyle\int_{s}^{T}%
\mathbb{1}_{[t_{n},T]}}\left( r\right) {f(r,{X_{r}^{t_{n},x_{n}},Y{%
_{r}^{t_{n},x_{n}}})}dr-\displaystyle\int_{s}^{T}U{_{r}^{t_{n},x_{n}}}%
dM_{r}^{X^{t_{n},x_{n}}}}\medskip \\
{-\displaystyle\int_{s}^{T}h(r,{X_{r}^{t_{n},x_{n}},Y{_{r}^{t_{n},x_{n}}}%
)dk_{r}^{t_{n},x_{n}}}}\,,\;s\in \lbrack 0,T].%
\end{array}
\label{generalized BSDE 2}
\end{equation}%
For the proof we will adapt the steps from the proof of \cite[Theorem 3.1]%
{bo-ca/04}.

\noindent \textit{Step 1.} The solutions satisfy the boundedness conditions
(for the proof see, e.g., \cite[Proposition 1.1]{pa-zh/98}):
\begin{equation*}
\begin{array}{l}
\mathbb{E}{\big(}\sup_{s\in \left[ t,T\right] }|Y_{s}^{t_{n},x_{n}}|^{2}{%
\big)+}\mathbb{E}{\displaystyle\int_{t}^{T}}||U_{s}^{t_{n},x_{n}}{\sigma ({%
X_{r}^{t_{n},x_{n}})}}||^{2}ds\leq C,\;\forall t\in \left[ 0,T\right]
,\;\forall n\in \mathbb{N}\medskip \\
\mathbb{E}{\big(}\sup_{s\in \left[ t,T\right] }|Y_{s}^{t,x}|^{2}{\big)+}%
\mathbb{E}{\displaystyle\int_{t}^{T}}||U_{s}^{t,x}{\sigma ({X_{r}^{t,x})}}%
||^{2}ds\leq C,\;\forall t\in \left[ 0,T\right] ,%
\end{array}%
\end{equation*}%
where $C>0$ is a constant not depending on $n$.

\noindent\textit{Step 2.} To obtain the tightness property with respect to
the \textrm{S}-topology it is sufficient to compute the conditional
variation $\mathrm{CV}_{T}$ (see definition (\ref{def cond var})) of the
processes $Y^{t_{n},x_{n}}$, $M^{t_{n},x_{n}}$ and $H^{t_{n},x_{n}}$
respectively; we recall the notation (\ref{def M and H}) for the quantities $%
M^{t_{n},x_{n}}$ and $H^{t_{n},x_{n}}$.

As in \cite[Theorem 3.1]{bo-ca/04}, after some easy computation we deduce
that there exists a constant $C>0$ independent of $n$, such that%
\begin{equation*}
\begin{array}{r}
\mathrm{CV}_{T}\left( Y^{t_{n},x_{n}}\right) +\mathbb{E}{\Big(}\sup_{s\in
\lbrack 0,T]}|Y_{s}^{t_{n},x_{n}}|{\Big)+}\mathbb{E}{\Big(}\sup_{s\in
\lbrack 0,T]}|M_{s}^{t_{n},x_{n}}|{\Big)+}\mathrm{CV}_{T}\left(
H^{t_{n},x_{n}}\right) \medskip \\
+\mathbb{E}{\Big(}\sup_{s\in \lbrack 0,T]}|H_{s}^{t_{n},x_{n}}|{\Big)}\leq
C,\;\forall n\in \mathbb{N}^{\ast }.%
\end{array}%
\end{equation*}

\noindent \textit{Step 3.} The above condition ensures (see \cite[Appendix A]%
{le/02} or \cite[Theorem 3.5]{bo-ca/04}) the tightness of the sequence $%
\left( Y^{t_{n},x_{n}},M^{t_{n},x_{n}},H^{t_{n},x_{n}}\right) $ with respect
to the \textrm{S}-topology. Therefore there exists a subsequence, still
denoted by $\left( Y^{t_{n},x_{n}},M^{t_{n},x_{n}},H^{t_{n},x_{n}}\right) $,
and a process $\left( {\bar{Y}},\,{\bar{M}},\,{\bar{H}}\right) \in \left(
\mathcal{D}\left( [0,T],{\mathbb{R}}^{k}\right) \right) ^{3}$ such that%
\begin{equation}
\left( {{X^{t_{n},x_{n}}}},{{K^{t_{n},x_{n}}}}%
,Y^{t_{n},x_{n}},M^{t_{n},x_{n}},H^{t_{n},x_{n}}\right)
\xrightarrow[\mathrm{U\times U\times S\times
S\times S}]{\;\;\;*\; \;\;}%
\left( X^{t,x},K^{t,x},{\bar{Y}},{\bar{M}},{\bar{H}}\right) ,
\label{converg of (X,Y)}
\end{equation}%
weakly on $({\mathcal{C}}([0,T],{\mathbb{R}}^{d}))^{2}\times \left( \mathcal{%
D}([0,T],{\mathbb{R}}^{k})\right) ^{3}$.

In order to pass to the limit in (\ref{generalized BSDE 2}) we use the
continuity of $f$, \cite[Corollary 2.11]{ja/97}, the Lipschitzianity of $h$,%
\begin{equation*}
{{k^{t_{n},x_{n}}%
\xrightarrow[u]{\;\;\;*\; \;\;}%
k^{t,x},}}
\end{equation*}%
and we apply \cite[Lemma 3.3]{bo-ca/04}; we precise that the conclusion of
this lemma is still true in the point $T$, hence there exists a countable
set $Q\subset \lbrack 0,T)$ such that, for any $s\in \lbrack 0,T]\setminus
Q~,$%
\begin{equation*}
{\bar{Y}}_{s}=g(X_{T}^{t,x})+\int_{s}^{T}{\mathbb{1}_{[t,T]}}\left( r\right)
f(r,X_{r}^{t,x},{\bar{Y}}_{r})dr-({\bar{M}}_{T}-{\bar{M}}_{s})-%
\int_{s}^{T}h(r,X_{r}^{t,x},{\bar{Y}}_{r})dk_{r}~.
\end{equation*}%
Since the processes ${\bar{Y}}$, ${\bar{M}}$ and ${\bar{H}}$ are c\`{a}dl%
\`{a}g, the above equality holds true for any $s\in \lbrack 0,T].$

We mention that ${M^{X^{t,x}}}$ and ${\bar{M}}$ are martingales with respect
to the same filtration. Indeed, ${\bar{M}_{s}}$ is ${\mathcal{F}}%
_{s}^{X^{t,x},{\bar{Y}},{\bar{M}}}$-adapted and, moreover, ${\bar{M}}$ is an
${\mathcal{F}}^{X^{t,x},{\bar{Y}},{\bar{M}}}$-martingale (for the proof see
\cite[Lemma A.1]{le/02}).

Let now $\psi_{s}$ be a bounded continuous mapping from ${\mathcal{C}}([0,s],%
{\mathbb{R}}^{d})\times\mathcal{D}([0,s],{\mathbb{R}}^{k})^{2}$, $\varphi\in
C^{\infty}({\mathbb{R}}^{d})$ and%
\begin{equation*}
L=\frac{1}{2}\sum_{i,j}(\sigma(\cdot)\,\sigma^{\ast}(\cdot))_{ij}(\cdot )%
\frac{{\partial}^{2}}{\partial x_{i}\partial x_{j}}+\sum_{i}b_{i}(\cdot )%
\frac{\partial}{\partial x_{i}}\,
\end{equation*}
be the infinitesimal generator of the diffusion process $X^{t_{n},x_{n}}$.

From It{\^{o}}'s formula we obtain that
\begin{equation*}
\varphi(X_{s}^{t_{n},x_{n}})-\varphi(x_{n})-\int_{t_{n}}^{s}L\varphi
(X_{r}^{t_{n},x_{n}})dr+\int_{t_{n}}^{s}\nabla%
\varphi(X_{r}^{t_{n},x_{n}})dK_{r}^{t_{n},x_{n}}
\end{equation*}
is a martingale.

Therefore, for any $0\leq s_{1}<s_{2}\leq T,$%
\begin{equation*}
\begin{array}{r}
\displaystyle{\mathbb{E}\Big[}\psi _{s_{1}}\left(
X^{t_{n},x_{n}},Y^{t_{n},x_{n}},M^{t_{n},x_{n}}\right) {\Big(}\varphi
(X_{s_{2}}^{t_{n},x_{n}})-\varphi (X_{s_{1}}^{t_{n},x_{n}})-\int_{s_{1}\vee
t_{n}}^{s_{2}\vee t_{n}}L\varphi (X_{r}^{t_{n},x_{n}})dr\medskip \\
\displaystyle+\int_{s_{1}\vee t_{n}}^{s_{2}\vee t_{n}}\nabla \varphi
(X_{r}^{t_{n},x_{n}})dK_{r}^{t_{n},x_{n}}{\Big)\Big]}=0,\;\forall n.%
\end{array}%
\end{equation*}%
It can be proved, using (\ref{converg of (X,Y)}), that%
\begin{equation*}
\begin{array}{l}
\lim_{n\rightarrow \infty }\displaystyle{\mathbb{E}\Big[}\psi _{s_{1}}\left(
X^{t_{n},x_{n}},Y^{t_{n},x_{n}},M^{t_{n},x_{n}}\right) {\Big(}\varphi
(X_{s_{2}}^{t_{n},x_{n}})-\varphi (X_{s_{1}}^{t_{n},x_{n}})-\int_{s_{1}\vee
t_{n}}^{s_{2}\vee t_{n}}L\varphi (X_{r}^{t_{n},x_{n}})dr{\Big)\Big]}\medskip
\\
=\displaystyle{\mathbb{E}\Big[\psi _{s_{1}}\left( X^{t,x},\bar{Y},\bar{M}%
\right) \Big(}\varphi (X_{s_{2}}^{t,x})-\varphi
(X_{s_{1}}^{t,x})-\int_{s_{1}\vee t}^{s_{2}\vee t}L\varphi (X_{r}^{t,x})dr{%
\Big)\Big].}%
\end{array}%
\end{equation*}%
On the other hand,%
\begin{equation*}
\begin{array}{l}
\lim_{n\rightarrow \infty }\displaystyle{\mathbb{E}\Big[\psi _{s_{1}}\left(
X^{t_{n},x_{n}},Y^{t_{n},x_{n}},M^{t_{n},x_{n}}\right) \int_{s_{1}\vee
t_{n}}^{s_{2}\vee t_{n}}\nabla \varphi
(X_{r}^{t_{n},x_{n}})dK_{r}^{t_{n},x_{n}}\Big]}\medskip \\
\displaystyle={\mathbb{E}\Big[\psi _{s_{1}}\left( X^{t,x},\bar{Y},\bar{M}%
\right) \int_{s_{1}\vee t}^{s_{2}\vee t}\nabla \varphi
(X_{r}^{t,x})dK_{r}^{t,x}\Big],}%
\end{array}%
\end{equation*}%
by \cite[Proposition 3.4]{za/08}.

Therefore%
\begin{equation*}
\begin{array}{r}
\displaystyle{\mathbb{E}\Big[\psi_{s_{1}}\left( X^{t,x},\bar{Y},\bar {M}%
\right) \Big(}\varphi(X_{s_{2}}^{t,x})-\varphi(X_{s_{1}}^{t,x})-\int_{s_{1}%
\vee t}^{s_{2}\vee t}L\varphi(X_{r}^{t,x})dr{\Big)\Big]}\medskip \\
\displaystyle+\int_{s_{1}\vee t}^{s_{2}\vee
t}\nabla\varphi(X_{r}^{t,x})dK_{r}^{t,x}{\Big)\Big]}=0.%
\end{array}%
\end{equation*}
Using It{\^{o}}'s formula we see that%
\begin{equation*}
{\mathbb{E}}\left[ \psi_{s_{1}}\left( X^{t,x},\bar{Y},\bar{M}\right) {%
\int_{s_{1}\vee t}^{s_{2}\vee t}}\nabla\varphi(X_{r}^{t,x})dM_{r}^{X^{t,x}}%
\right] =0
\end{equation*}
and therefore $M^{X^{t,x}}$ is a ${\mathcal{F}}^{X^{t,x},{\bar{Y}},{\bar{M}}%
} $-martingale.

Now since $Y^{t,x}$ and ${U{^{t,x}}}$ are ${\mathcal{F}}^{X^{t,x}}$-adapted,
${M^{t,x}:=\int_{t}^{\cdot}U{_{r}^{t,x}}dM_{r}^{X^{t,x}}}$ is also ${%
\mathcal{F}}^{X^{t,x},{\bar{Y}},{\bar{M}}}$-martingale.

Let us take $0\leq s_{1}\leq s_{2}\leq T$. It{\^{o}}'s formula yields%
\begin{equation*}
\begin{array}{l}
\displaystyle|Y_{s_{1}}^{t,x}-{\bar{Y}}_{s_{1}}|^{2}+\left( [M^{t,x}-{\bar {M%
}}]_{s_{2}}-[M^{t,x}-{\bar{M}}]_{s_{1}}\right) =|Y_{s_{2}}^{t,x}-{\bar{Y}}%
_{s_{2}}|^{2}\medskip \\
\displaystyle\quad+2\int_{s_{1}}^{s_{2}}\left\langle Y_{r}^{t,x}-{\bar{Y}}%
_{r},f(r,X_{r}^{t,x},{Y}_{r}^{t,x})-f(r,X_{r}^{t,x},{\bar{Y}}%
_{r})\right\rangle dr\medskip \\
\displaystyle\quad+2\int_{s_{1}}^{s_{2}}\left\langle Y_{r}^{t,x}-{\bar{Y}}%
_{r},h(r,X_{r}^{t,x},{Y}_{r}^{t,x})-h(r,X_{r}^{t,x},{\bar{Y}}%
_{r})\right\rangle dA_{r}^{t,x}\medskip \\
\displaystyle\quad-2\int_{s_{1}}^{s_{2}}\left\langle Y_{r}^{t,x}-{\bar{Y}}%
_{r},d(M_{r}^{t,x}-\bar{M}_{r})\right\rangle \medskip \\
\displaystyle\leq|Y_{s_{2}}^{t,x}-{\bar{Y}}_{s_{2}}|^{2}+2\alpha\vee \beta%
\displaystyle\int_{s_{1}}^{s_{2}}|Y_{r}^{t,x}-{\bar{Y}_{r}}%
|^{2}d(r+A_{r}^{t,x})-2\int_{s_{1}}^{s_{2}}\left\langle Y_{r}^{t,x}-{\bar{Y}}%
_{r},d(M_{r}^{t,x}-\bar{M}_{r})\right\rangle .%
\end{array}%
\end{equation*}
since $\displaystyle\int_{t}^{\cdot}\left\langle Y_{r}^{t,x}-{\bar{Y}}%
_{r},d(M_{r}^{t,x}-\bar{M}_{r})\right\rangle $ is a ${\mathcal{F}}^{X^{t,x},{%
\bar{Y}},{\bar{M}}}$-martingale.

Hence, from a generalized Gronwall lemma (see, e.g., \cite[Lemma 12]%
{ma-ra/07}), by taking $s_{2}=T$, we deduce the identification%
\begin{equation*}
Y^{t,x}={\bar{Y}\;\;}\text{and \ }M^{t,x}={\bar{M}.}
\end{equation*}%
\hfill
\end{proof}

\subsection{Proof of Theorem \protect\ref{Main result 1}}

Let us denote%
\begin{equation}
u^{n}(t,x):=Y_{t}^{t,x,n}\text{\quad and\quad}u(t,x):=Y_{t}^{t,x}.
\label{def of u^n and u}
\end{equation}
Hence, $u^{n}$ and $u$ are deterministic functions since $Y^{t,x,n}\;$is
adapted with respect to the filtration generated by $X^{t,x,n}$ and $%
Y^{t,x}\;$is adapted with respect to the filtration generated by $X^{t,x}$.

First we prove that the functions $u^{n}:[0,T]\times \mathbb{R}%
^{d}\rightarrow \mathbb{R}^{d}$ and $u:[0,T]\times \bar{G}\rightarrow
\mathbb{R}^{d}$ defined by (\ref{def of u^n and u}) are continuous. We will
show only that the function $u$ is continuous. Let $\left(
t_{n},x_{n}\right) \rightarrow \left( t,x\right) \in \left[ 0,T\right]
\times \bar{G}$, as $n\rightarrow \mathbb{\infty }$. From the proof of
Proposition \ref{cont of Y}, we can extract a subsequence still denoted $%
\left( t_{n},x_{n}\right) $, such that
\begin{equation*}
\left( {{X^{t_{n},x_{n}}}},{{K^{t_{n},x_{n}}}}%
,Y^{t_{n},x_{n}},M^{t_{n},x_{n}}\right)
\xrightarrow[\mathrm{U\times U\times S\times
S}]{\;\;\;*\; \;\;}%
\left( X^{t,x},K^{t,x},{Y}^{t,x},{M}^{t,x}\right) .
\end{equation*}%
We know from \cite[Lemma 3.3]{bo-ca/04} applied for $t=T$, that%
\begin{equation*}
{\int_{0}^{T}h(r,{X_{r}^{t_{n},x_{n}},Y{_{r}^{t_{n},x_{n}}}%
)dk_{r}^{t_{n},x_{n}}}}\rightarrow {\int_{0}^{T}h(r,{X_{r}^{t,x},Y{_{r}^{t,x}%
})dk_{r}^{t,x}~}}\text{in law, as }n\rightarrow \infty .
\end{equation*}%
Using \cite[Remark 2.4]{ja/97}, we see that ${M}_{T}^{t_{n},x_{n}}%
\xrightarrow[]{\;\;\; \;}%
{M}_{T}^{t,x}$ in law, since ${M}^{t_{n},x_{n}}%
\xrightarrow[\mathrm{S}]{\;\;\;*\; \;\;}%
{M^{t,x}}$.

Hence we can pass to the limit in%
\begin{equation*}
\begin{array}{r}
u(t_{n},x_{n})=Y_{t_{n}}^{t_{n},x_{n}}=Y_{0}^{t_{n},x_{n}}{=g(X}%
_{T}^{t_{n},x_{n}}){+\displaystyle\int_{0}^{T}\mathbb{1}_{[t_{n},T]}}\left(
r\right) {f(r,{X_{r}^{t_{n},x_{n}},Y{_{r}^{t_{n},x_{n}}})}dr-M}%
_{T}^{t_{n},x_{n}}\medskip \\
{-\displaystyle\int_{0}^{T}h(r,{X_{r}^{t_{n},x_{n}},Y{_{r}^{t_{n},x_{n}}}%
)dk_{r}^{t_{n},x_{n}}}}%
\end{array}%
\end{equation*}%
and, as in the the proof of Proposition \ref{cont of Y}, we deduce that the
limit of $u(t_{n},x_{n})$ is%
\begin{equation*}
\begin{array}{r}
{g(X}_{T}^{t,x}){+\displaystyle\int_{0}^{T}\mathbb{1}_{[t,T]}}\left(
r\right) {f(r,{X_{r}^{t,x},Y{_{r}^{t,x}})}dr-M}_{T}^{t,x}{-\displaystyle%
\int_{0}^{T}h(r,{X_{r}^{t,x},Y{_{r}^{t,x}})dk_{r}^{t,x}}}\medskip \\
{{=Y_{0}^{t,x}}}=Y_{t}^{t,x}=u\left( t,x\right) .%
\end{array}%
\end{equation*}%
It is easy to show that, even if $b$ and $\sigma $ are only continuous
functions, the proof from \cite[Theorem 4.3]{pa-zh/98} (see also \cite[%
Theorem 3.2]{pa/99} for nonreflecting case) still works in order to show
that the functions $u^{n}$ and $u$ defined by (\ref{def of u^n and u}) are
viscosity solutions of the PDEs (\ref{PDE Neumann approx}) and (\ref{PDE
Neumann}) respectively.

Finally, as a consequence of Proposition \ref{result_Boufoussi} we deduce
the solution $u$ of the deterministic system (\ref{PDE Neumann}) is
approximated by the functions $u^{n}$, i.e.%
\begin{equation*}
\lim_{n\rightarrow\infty}u^{n}(t,x)=u(t,x),\;\forall\left( t,x\right)
\in\lbrack0,T]\times\bar{G}.
\end{equation*}

\noindent \textbf{Acknowledgement.} \textit{The authors would like to thank
the referee for the remarks and comments which have led to a significant
improvement of the paper. The authors L. Maticiuc and A. Z\u{a}linescu thank
the IMATH laboratory of Universit\'{e} du Sud Toulon Var for its hospitality.%
}

\end{document}